\documentclass[10pt,reqno]{amsart}
\usepackage{bbm}
\usepackage{mathrsfs}
\usepackage{amsfonts} 
\usepackage[dvipsnames,usenames]{color}
\usepackage{graphicx,latexsym,bm,amsmath,amssymb,verbatim,multicol,lscape}
\newtheorem{thm}{Theorem} [section]

\newtheorem{lem}[thm]{Lemma}

\theoremstyle{definition}

\theoremstyle{remark}

\numberwithin{equation}{section}

\begin{document}
\title{A generalization of a theorem of Nagell}
\author[Y.L. Feng, S.F. Hong, X. Jiang and Q.Y. Yin]
{Yulu Feng, Shaofang Hong$^{*, \dagger}$, Xiao Jiang and Qiuyu Yin\\
\small{Mathematical College, Sichuan University, Chengdu 610064, P.R. China\\
e-mails: yulufeng17@126.com (Y.L. Feng); sfhong@scu.edu.cn, s-f.hong@tom.com,
hongsf02@yahoo.com (S.F. Hong); 422040631@qq.com (X. Jiang); 
yinqiuyu26@126.com (Q.Y. Yin)}}
\thanks{$^*$Corresponding author}
\thanks{$^{\dagger}$The research was supported partially by National Science
Foundation of China Grant \#11771304
and by the Fundamental Research Funds for the Central Universities.}
\subjclass[2000]{primary 11N13, 11B25, 11B83, 11B75}
\keywords{$p$-adic valuation, arithmetic progression, reciprocal
power sum, Bertrand's postulate, integrality}
\begin{abstract}
Let $n$ be a positive integer. In 1915, Theisinger proved that if
$n\ge 2$, then the $n$-th harmonic sum $\sum_{k=1}^n\frac{1}{k}$
is not an integer. Let $a$ and $b$ be positive integers. In 1923,
Nagell extended Theisinger's theorem by showing that the reciprocal
sum $\sum_{k=1}^{n}\frac{1}{a+(k-1)b}$ is not an integer if $n\ge 2$.
In 1946, Erd\H{o}s and Niven proved a theorem of a similar nature that
states that there is only a finite number of integers $n$ for
which one or more of the elementary symmetric functions of
$1,1/2, ..., 1/n$ is an integer. In this paper, we present
a generalization of Nagell's theorem. In fact, we show that
for arbitrary $n$ positive integers $s_1, ..., s_n$
(not necessarily distinct and not necessarily monotonic),
the following reciprocal power sum
$$\sum\limits_{k=1}^{n}\frac{1}{(a+(k-1)b)^{s_{k}}}$$
is never an integer if $n\ge 2$. The proof of our
result is analytic and $p$-adic in character.
\end{abstract}
\maketitle

\section{Introduction}
Let $\mathbb{Z}$, $\mathbb{Z}^+$ and $\mathbb{Q}$ be the set
of integers, the set of positive integers and the set of
rational numbers, respectively. Let $n\in \mathbb{Z}^+$.
More than one hundred years ago, Theisinger \cite{[T]} proved
that the $n$-th harmonic sum $1+\frac{1}{2}+...+\frac{1}{n}$
is never an integer if $n>1$. Eight years later, Nagell
\cite{[N]} extended Theisinger's theorem from the sequence
of positive integers to general arithmetic progressions
by showing that if $a$ and $b$ are positive integers and $n\ge 2$,
then the reciprocal sum $\sum_{i=0}^{n-1}\frac{1}{a+bi}$
is never an integer. Then Erd\H{o}s and Niven \cite{[EN]}
generalized Nagell's theorem by establishing a similar result
on the integrality of the elementary symmetric functions
of $\frac{1}{a}, \frac{1}{a+b}, ..., \frac{1}{a+(n-1)b}$.
In the recent years, Erd\H{o}s and Niven's result was
extended to arbitrary polynomial sequences, see \cite{[CT]},
\cite{[HW]}, \cite{[LHQW]}, \cite{[WH]} and \cite{[YHYQ]}.

Throughout, we let $a$ and $b$ be positive integers. By
$(\mathbb{Z}^+)^{\infty}$ we denote the set of all infinite
sequences $\{s_i\}_{i=1}^{\infty}$ of positive integers
(note that all the $s_i$ are not necessarily distinct
and not necessarily monotonic). For any given
$\mathcal{S}=\{s_i\}_{i=1}^{\infty}\in
(\mathbb{Z}^+)^{\infty}$, we let
$\mathcal{S}_n:=\{s_1, ..., s_n\}.$
Associated to the infinite sequence $\mathcal{S}$
of positive integers, one can form an infinite sequence
$\{H_{a, b}(\mathcal{S}_n)\}_{n=1}^{\infty}$ of positive
rational fractions with $H_{a, b}(\mathcal{S}_n)$
being the $n$-th reciprocal power sum defined as follows:
$$H_{a, b}(\mathcal{S}_n)
:=\sum\limits_{k=1}^{n}\frac{1}{(a+(k-1)b)^{s_k}}.$$
If $s_i=1$ for all integers $i\ge 1$, then we write
$H_{a, b}(n)$ for $H_{a, b}(\mathcal{S}_n)$.
By Nagell's theorem \cite{[N]} and the main result
of \cite{[LHQW]}, we know that if $n\ge 2$ and $s_1=...=s_n$,
then $H_{a, b}(\mathcal{S}_n)$ is never an integer. In 2017,
Yang, Li, Feng and Jiang \cite{[YLFJ]} showed an extension of
Theisinger's theorem that states that
$H_{1, 1}(\mathcal{S}_n)=1+\frac{1}{2^{s_2}}+...+\frac{1}{n^{s_n}}$
is never an integer if $n\ge 2$.

In this paper, we address the problem of integrality of the
$n$-th reciprocal power sum $H_{a, b}(\mathcal{S}_n)$. In fact,
we present the following generalization of Nagell's theorem
\cite{[N]}.

\begin{thm}\label{thm1}
For any infinite sequence $\mathcal{S}$ of positive integers
and arbitrary positive integers $a, b$, the $n$-th reciprocal
power sum $H_{a, b}(\mathcal{S}_n)$ is never an integer if $n\ge 2$.
\end{thm}

Letting $s_1=...=s_n=1$ in Theorem 1.1 gives us Nagell's theorem
\cite{[N]}, and picking $a=b=1$ in Theorem 1.1 yields the result
of Yang, Li, Feng and Jiang \cite{[YLFJ]}.
The proof of Theorem 1.1 is analytic and $p$-adic in character.

This paper is organized as follows. First of all, in Section 2,
we show some preliminary lemmas which are needed in the proof of
Theorem \ref{thm1}. Then in Section 3, we give the proof of
Theorem \ref{thm1}.

Throughout, we always let $a, b$ and $n$ be positive integers with
$n\ge 2$. As usual, for any prime $p$ and for any integer $m$,
we let $v_p(m)$ stand for the {\it $p$-adic valuation}
of $m$, i.e., $v_p(m)$ is the biggest nonnegative integer
$r$ with $p^r$ dividing $m$. If $x=\frac{m_1}{m_2}$,
where $m_1$ and $m_2$ are integers and $m_2\ne 0$,
then we define $v_p(x):=v_p(m_1)-v_p(m_2)$.

\section{Auxiliary lemmas}

In this section, we present several auxiliary lemmas
that are needed in the proof of Theorem {\ref{thm1}}.
Clearly, $H_{a, b}(\mathcal{S}_n)>0$.

\begin{lem}\label{lem2.1}
Let $a=b>\frac{n}{2}$. Then $H_{a, b}(\mathcal{S}_n)<1$.
\end{lem}

\begin{proof}
First, let $n=2$ or $3$. Then $a\ge 2$ since $a>\frac{n}{2}$.
Hence
$$H_{a, b}(\mathcal{S}_n)\le H_{a, b}(n)
=\sum\limits_{k=1}^{n}\frac{1}{ka}
\le \sum\limits_{k=1}^{3}\frac{1}{2k}=\frac{11}{12}<1$$
as expected.

Consequently, let $n=4$. Then $a\ge 3$. It follows that
$$H_{a, b}(\mathcal{S}_n)\le \sum\limits_{k=1}^{n}\frac{1}{ka}
\le \sum\limits_{k=1}^{4}\frac{1}{3k}=\frac{25}{36}<1$$
as desired.

Now let $n\ge 5$. Noting that for any positive integer $m$,
$$
(m+1)\sum\limits_{k=1}^m\frac{1}{k}
-m\sum\limits_{k=1}^{m+1}\frac{1}{k}
=\sum\limits_{k=1}^m\frac{1}{k}-\frac{m}{m+1}>0
$$
implying that
$$
\frac{1}{m}\sum\limits_{k=1}^m\frac{1}{k}>
\frac{1}{m+1}\sum\limits_{k=1}^{m+1}\frac{1}{k},
$$
one derives that
$$H_{a, b}(\mathcal{S}_n)\le \sum\limits_{k=1}^{n}\frac{1}{ka}
<\frac{2}{n}\sum\limits_{k=1}^{n}\frac{1}{k}
\le \frac{2}{5}\sum\limits_{k=1}^{5}\frac{1}{k}
=\frac{137}{150}<1$$
as required. So Lemma \ref{lem2.1} is proved.
\end{proof}

\begin{lem}\label{lem2.2}
Let $2\le n<\frac{a}{b}+1$. Then $H_{a, b}(\mathcal{S}_n)<1$.
\end{lem}

\begin{proof}
First, it follows from the hypothesis that
$2\le n<\frac{a}{b}+1$ that $a>b\ge 1$.

Firs of all, we let $a=2$. Then $b=1$ and $n=2$. So
$$H_{a, b}(\mathcal{S}_n)\le H_{2, 1}(2)=\frac{1}{2}+\frac{1}{3}<1$$
as desired.

Finally, let $a\ge 3$. Then
\begin{align*}
H_{a, b}(\mathcal{S}_n)&\le \sum\limits_{k=1}^{n}\frac{1}{a+(k-1)b} \\
&\le \frac{1}{a}+\frac{1}{a+b}+(n-2)\frac{1}{a+2b} \\
&<\frac{1}{a}+\frac{1}{a+b}+(\frac{a}{b}-1)\frac{1}{a+2b}\\
&\le \frac{1}{a}+\frac{1}{a+1}+(a-1)\frac{1}{a+2}\\
&=1-\frac{a^2-2a-2}{a(a+1)(a+2)}<1
\end{align*}
as expected. Hence Lemma \ref{lem2.2} is proved.
\end{proof}

\begin{lem}\label{lem2.3}
Let $a\ne b$ and $b\in \{\gcd (a, b), 2\gcd (a, b)\}$.
If $2\gcd (a, b)^2>a+(n-1)b$, then $H_{a, b}(\mathcal{S}_n)<1$.
\end{lem}

\begin{proof} Let $d:=\gcd (a, b)$. If $a=d$, then $b=2d$
and $d>\frac{a+(n-1)b}{2d}=\frac{2n-1}{2}$. Hence
\begin{align*}
H_{a, b}(\mathcal{S}_n)
&\le H_{d, 2d}(n) \\
&=\sum\limits_{k=1}^{n}\frac{1}{d+2d(k-1)} \\
&=\frac{1}{d}\sum\limits_{k=1}^{n}\frac{1}{2k-1} \\
&<\frac{2}{2n-1}\Big(1+\frac{n-1}{3}\Big)\\
&=\frac{1}{3}+\frac{5}{6n-3}\le\frac{8}{9}<1
\end{align*}
as desired.

If $a\ge 2d$, then $b\ge d$ and $d>\frac{a+(n-1)b}{2d}
\ge \frac{n+1}{2}$. Therefore
$$H_{a, b}(\mathcal{S}_n)\le \sum\limits_{k=1}^{n}\frac{1}{2d+d(k-1)}
=\frac{1}{d}\sum\limits_{k=1}^{n}\frac{1}{k+1}
<\frac{2}{n+1}\times \frac{n}{2}<1 $$
as expected. So Lemma \ref{lem2.3} is proved.
\end{proof}

\begin{lem}\label{lem2.4}
Let $n$ be an integer such that
$2\le n<1+\frac{a+b}{b}e^{b(1-\frac{1}{a}-\frac{1}{a+b})}-\frac{a}{b}$.
Then $H_{a, b}(\mathcal{S}_n)<1$.
\end{lem}

\begin{proof}
Since
$$n-1<\frac{a+b}{b}e^{b(1-\frac{1}{a}-\frac{1}{a+b})}-\frac{a}{b}:=c$$
and
$$\frac{1}{a+(k-1)b}\le \int_{k-2}^{k-1}\frac{1}{a+bt}dt \eqno(1)$$
if $k\ge 3$, one deduces that
\begin{align*}
H_{a, b}(\mathcal{S}_n)
&\le H_{a, b}(n) \\
&=\sum\limits_{k=1}^{n}\frac{1}{a+(k-1)b} \\
&\le \frac{1}{a}+\frac{1}{a+b}+\int_{1}^{n-1}\frac{1}{a+bt}dt \\
&=\frac{1}{a}+\frac{1}{a+b}+\frac{1}{b}\log\frac{a+(n-1)b}{a+b} \\
&<\frac{1}{a}+\frac{1}{a+b}+\frac{1}{b}\log\frac{a+cb}{a+b}=1
\end{align*}
as one expects. Lemma \ref{lem2.4} is proved.
\end{proof}

\begin{lem}\label{lem2.5}
Let $a>b\ge 3$ and $(a, b)\ne (4, 3)$. Then
$$\frac{a+b}{2b}e^{b(1-\frac{1}{a}-\frac{1}{a+b})}-\frac{a}{2b}
+\frac{1}{2}>\frac{3b+\sqrt{9b^2+12(a-b)}}{2}. $$
\end{lem}

\begin{proof}
First of all, since $a>b\ge 3$ and $(a, b)\ne (4, 3)$, we have
$a\ge 5$ and
$$3b+\frac{a}{b}>\frac{3b+\sqrt{9b^2+12(a-b)}}{2}. $$
So to prove Lemma 2.5, it suffices to show that $g(a, b)>0$
with the two-variable function $g(x, y)$ being defined by
$$g(x, y):=\frac{x+y}{2y}e^{y(1-\frac{1}{x}-\frac{1}{x+y})}
-3y-\frac{3x}{2y}+\frac{1}{2}.$$

Evidently, one has
$$g(x, y)=\frac{x}{2y}\big(e^{y(1-\frac{1}{x}-\frac{1}{x+y})}-3\big)
+\frac{1}{2}e^{y(1-\frac{1}{x}-\frac{1}{x+y})}-3y+\frac{1}{2}$$
and
$e^{y(1-\frac{1}{x}-\frac{1}{x+y})}\ge e^{\frac{81}{40}}\approx 7.576110945$
when $x\ge 5$ and $y\ge 3$.
Hence $g(x, y)$ increases as $x$ increases in the interval
$[5, \infty)$ when $y\ge 3$.

Let $a=5$ and $b=3$. Then one can directly compute and find that
$$\frac{a+b}{2b}e^{b(1-\frac{1}{a}-\frac{1}{a+b})}-\frac{a}{2b}
+\frac{1}{2}\approx 9.768147927>9.7$$
and
$$\frac{3b+\sqrt{9b^2+12(a-b)}}{2}\approx 9.623475385<9.7.$$
So Lemma 2.5 is true when $a=5$ and $b=3$.

Let $a\ge 6$ and $b=3$. Then
$$g(a, b)=g(a, 3)\ge g(6, 3)\approx 1.59370755>0. $$

Finally, let $a>b\ge 4$. Then $a\ge b+1$. Noticing that $g(x, y)$
is increasing in the variable $x\in [5, \infty)$ when $y\ge 3$
and $1-\frac{1}{b+1}-\frac{1}{2b+1}\ge\frac{31}{45}$ if $b\ge 4$, we obtain that
\begin{align*}
g(a, b)
&\ge g(b+1, b) \\
&=(1+\frac{1}{2b})e^{b(1-\frac{1}{b+1}-\frac{1}{2b+1})}
-3b-\frac{3}{2b}-1 \\
&>e^{\frac{31}{45}b}-3b-\frac{11}{8} \\
&\ge e^{\frac{124}{45}}-12-\frac{11}{8} \\
&\approx 2.35477725>0,
\end{align*}
where the last second inequality holds due to the fact
$e^{\frac{31}{45}b}-3b-\frac{11}{8}$
increases as $b\in[4, \infty)$ increases.
Hence Lemma \ref{lem2.5} is proved.
\end{proof}

\begin{lem}\label{lem2.6}
Let $b>a\ge 2$ and $4a+b\ge 18$. Then
$$\frac{a+b}{2b}e^{b(1-\frac{1}{a}-\frac{1}{a+b})}-\frac{a}{2b}
+\frac{1}{2}>3b. $$
\end{lem}

\begin{proof}
First, we introduce the two-variable function $h(x, y)$
as follows:
$$h(x, y):=\frac{x+y}{2y}e^{y(1-\frac{1}{x}-\frac{1}{x+y})}
-\frac{x}{2y}+\frac{1}{2}-3y. $$
Then
$$h(x, y)=\frac{x}{2y}\big(e^{y(1-\frac{1}{x}-\frac{1}{x+y})}-1\big)
+\frac{1}{2}e^{y(1-\frac{1}{x}-\frac{1}{x+y})}
+\frac{1}{2}-3y. $$
Hence $h(x, y)$ increases as $x$
increases in the interval $[2, \infty)$ when $y>0$.
Therefore, for $b\ge 10$ and $a\ge 2$, we have
\begin{align*}
h(a, b)
&\ge h(2, b) \\
&=(\frac{1}{b}+\frac{1}{2})e^{\frac{b^2}{2(b+2)}}-3b
+\frac{1}{2}-\frac{1}{b} \\
&> \frac{1}{2}e^{\frac{b^2}{2(b+2)}}-3b \\
&=\frac{1}{2}e^{\frac{b}{2}(1-\frac{2}{b+2})}-3b \\
&\ge \frac{1}{2}e^{\frac{5}{12}b}-3b \\
&\ge \frac{1}{2}e^{\frac{25}{6}}-30 \\
&\approx 2.25004654>0.
\end{align*}
Thus Lemma 2.6 is true when $b\ge 10$.

Now let $b\le 9$. Then there are exactly $\binom {8}{2}=28$ pairs
$(a, b)$ satisfying $2\le a<b\le 9$ and we can calculate the values
of $h(a, b)$ one by one. By some computations, we find that
$h(a, b)<0$ if $(a, b)$ belongs to the following set
$$R:=\{(2, 3), (2, 4), (2, 5), (2, 6), (2, 7), (2, 8),
(2, 9), (3, 4), (3, 5)\}, \eqno(2)$$
and $h(a, b)>0$ if $(a, b)$ takes the remaining 19 pairs.
In other words, $h(a, b)<0$ if $4a+b\le 17$,
and $h(a, b)>0$ if $4a+b\ge 18$.

The proof of Lemma \ref{lem2.6} is complete.
\end{proof}

\begin{lem}\label{lem2.7}
Let $a\ge 2$, $b\ge 3$ and $a\ne b$ such that $H_{a, b}(\mathcal{S}_n)\ge 1$.
Then there is a prime $p$ such that exactly one term in $\{a+(k-1)b\}_{k=1}^{n}$
is divisible by $p$, or $p\in(\max(\frac{n}{2}, \alpha), n]$ with
$$\alpha:={\left\{\begin{array}{rl}
\frac{3b+\sqrt{9b^2+12(a-b)}}{2}, \ \ if \ \ a>b;\\
3b, \ \ if \ \ a<b.\\
\end{array}\right.} \eqno(3)$$
\end{lem}

\begin{proof}
Since $H_{a, b}(\mathcal{S}_n)\ge 1$, by Lemma \ref{lem2.4}, one has
$$n\ge \frac{a+b}{b}e^{b(1-\frac{1}{a}-\frac{1}{a+b})}-\frac{a}{b}+1.\eqno(4)$$
On the other hand, by Bertrand's postulate, there is a prime $p\in (\frac{n}{2}, n]$.

If $a>b\ge 3$ and $(a, b)\ne (4, 3)$, then by (4) and Lemma \ref{lem2.5}, one has
$$p>\frac{n}{2}\ge \frac{a+b}{2b}e^{b(1-\frac{1}{a}-\frac{1}{a+b})}-\frac{a}{2b}
+\frac{1}{2}>\frac{3b+\sqrt{9b^2+12(a-b)}}{2}=\alpha$$
as desired. So Lemma 2.7 is true in this case.

If $(a, b)=(4, 3)$, then
$$p>\frac{n}{2}\ge \frac{a+b}{2b}e^{b(1-\frac{1}{a}-\frac{1}{a+b})}-\frac{a}{2b}
+\frac{1}{2}=\frac{7}{6}e^{\frac{51}{28}}-\frac{1}{6}
\approx 7.044128639.$$
But $p$ is a prime and $a=4, b=3$. So
$p\ge 11>9.321825380\approx \frac{3b+\sqrt{9b^2+12(a-b)}}{2}=\alpha.$
Hence Lemma 2.7 holds in this case.

If $2\le a<b$ and $4a+b\ge 18$, then by (4) and Lemma \ref{lem2.6},
$$p>\frac{n}{2}\ge \frac{a+b}{2b}e^{b(1-\frac{1}{a}-\frac{1}{a+b})}-\frac{a}{2b}
+\frac{1}{2}>3b=\alpha, $$
which means the truth of Lemma 2.7 in this case.

Now let $2\le a<b$ and $4a+b\le 17$. Then one can easily derive that the set
of all the pairs $(a, b)$ equals the set $R$ given in (2). Clearly, if $n\ge 6b$,
then $p>\frac{n}{2}\ge 3b=\alpha$, as Lemma 2.7 claimed.
In what follows, we let $n\le 6b-1$. First, we assert that
$(a, b)\not\in \{(2, 9), (3, 5)\}$. Otherwise, one has $(a, b)=(2, 9)$ or $(3, 5)$.
But $n\le 6b-1$ and a direct computation gives us that
$$H_{2, 9}(53)=$$
$$\frac{70773412390639611995377611407286048258428112436665192898048216184826913}
{72993325114428717314530010109453157943718362576137475837761483670815360} <1$$
and
$$
H_{3, 5}(29)= \frac{61763030785793910862459859011}
{62877130769344946602672156032}<1.
$$
Then it follows that for $(a, b)=(2, 9)$ or $(3, 5)$, we have
$$H_{a, b}(\mathcal{S}_n)\le H_{a, b}(n) \le H_{a, b}(6b-1)
=\sum\limits_{k=1}^{6b-1}\frac{1}{a+(k-1)b}<1.$$
This contradicts with the assumption
$H_{a, b}(\mathcal{S}_n)\ge 1$. The assertion is true.

In the following, we show that if
$$(a, b)\in R\setminus\{(2,9), (3,5)\}
=\{(2, 3), (2, 4), (2, 5), (2, 6),
(2, 7), (2, 8), (3, 4)\},$$
then there is a prime $p$ with exactly one term in
$\{a+(k-1)b\}_{k=1}^{n}$ being divisible by $p$.
This will be done in what follows.

Let $(a, b)=(2, 3)$. Then
$$H_{a, b}(\mathcal{S}_n)\le H_{a, b}(n)\le H_{2, 3}(5)
=\sum_{k=1}^5\frac{1}{3k-1}=\frac{3041}{3080}<1$$
if $n\le 5$. So we must have $n\ge 6$. But $n\le 6b-1=17$. That is, $6\le n\le 17$.
We can choose $p=17=2+5\times 3\in\{a+(k-1)b\}_{k=1}^{n}$.
So Lemma 2.7 is proved in this case.

Let $(a, b)=(2, 4)$. Then
$$H_{a, b}(\mathcal{S}_n)\le H_{a, b}(n)\le H_{2, 4}(7)
=\sum_{k=1}^7\frac{1}{4k-2}=\frac{88069}{90090}<1$$
if $n\le 7$. So $n\ge 8$. But $n\le 6b-1=23$. Namely, $8\le n\le 23$.
We pick $p=13$ for $8\le n\le 19$ since
$2p=2+6\times 4\in \{a+(k-1)b\}_{k=1}^{n}$,
and $p=37$ for $20\le n\le 23$ since
$2p=2+18\times 4\in \{a+(k-1)b\}_{k=1}^{n}$.

Let $(a, b)=(2, 5)$. Then
$$H_{a, b}(\mathcal{S}_n)\le H_{a, b}(n)\le H_{2, 5}(11)
=\sum_{k=1}^{11}\frac{1}{5k-3}=\frac{3616405543}{3652567776}<1$$
if $n\le 11$. Thus $12\le n\le 6b-1=29$. Picking $p=47=2+9\times 5$
gives us the desired result.

Let $(a, b)=(2, 6)$. Then
$$H_{a, b}(\mathcal{S}_n)\le H_{a, b}(n)\le H_{2, 6}(17)
=\sum_{k=1}^{17}\frac{1}{6k-4}=\frac{2038704876507433}{2053923842370400}<1$$
if $n\le 17$. Therefore, $18\le n\le 6b-1=35$. Choosing $p=43$ gives us
the required result since
$2p=2\times 43=2+14\times 6\in\{a+(k-1)b\}_{k=1}^{n}$.

Let $(a, b)=(2, 7)$. Then
$$H_{a, b}(\mathcal{S}_n)\le H_{a, b}(n)\le H_{2, 7}(27)
=\sum_{k=1}^{27}\frac{1}{7k-5}=\frac{1237220537370712858171751080193}
{1241931941639876926714128796800}<1$$
if $n\le 27$. So $28\le n\le 6b-1=41$. We let $p=163=2+23\times 7$
as one desires.

Let $(a, b)=(2, 8)$. Then
$$H_{a, b}(\mathcal{S}_n)\le H_{a, b}(n)\le H_{2, 8}(43)
=\sum_{k=1}^{43}\frac{1}{8k-6}$$
$$=\frac{3367642441187401373402635301280230085911262853}
{3374879226092212539809802981326899789745565750}<1$$
if $n\le 43$. This means that $44\le n\le 6b-1=47$.
Then $p=157$ leads to what we want since
$2p=2+39\times 8\in\{a+(k-1)b\}_{k=1}^{n}$.

Let $(a, b)=(3, 4)$. Then
$$H_{a, b}(\mathcal{S}_n)\le H_{a, b}(n)\le H_{3, 4}(18)
=\sum_{k=1}^{18}\frac{1}{4k-1}=\frac{17609244113383887374}
{17652709515783236895}<1$$
if $n\le 18$. That is, $19\le n\le 6b-1=23$.
At this moment, taking $p=71=3+17\times 4$, the desired result follows.

This finishes the proof of Lemma \ref{lem2.7}.
\end{proof}

\begin{lem}\label{lem2.8}
Let $a$ and $b$ be positive integers. If $p$ is a prime
and exactly one term in $\{a+(k-1)b\}_{k=1}^{n}$
is divisible by $p$, then $v_p(H_{a, b}(\mathcal{S}_n))<0$.
\end{lem}

\begin{proof}
Let $n_0$ be an integer such that $1\le n_0\le n$ and $p|(a+(n_0-1)b)$
and $p\nmid (a+(k-1)b)$ for any integer $k$ with $k\ne n_0$ and
$1\le k\le n$. Then
$$v_p\big(\frac{1}{a+(n_0-1)b}\big)\le -1$$
and
$$v_p\Big(\sum\limits_{k=1 \atop k\ne n_0}^{n}
\frac{1}{(a+(k-1)b)^{s_k}}\Big)\ge 0. $$
It follows from the isosceles triangle principle
(see, for example, \cite{[K]}) that
$$
v_p(H_{a, b}(\mathcal{S}_n))=v_p\Big(\sum\limits_{k=1 \atop k\ne n_0}^{n}
\frac{1}{(a+(k-1)b)^{s_k}}+\frac{1}{(a+(n_0-1)b)^{s_{n_0}}}\Big)
$$
$$
=\min\Big(v_p\Big(\sum\limits_{k=1 \atop k\ne n_0}^{n}
\frac{1}{(a+(k-1)b)^{s_k}}\Big), v_p\Big(\frac{1}{(a+(n_0-1)b)^{s_{n_0}}}\Big)\Big)
$$
$$
=v_p\Big(\frac{1}{(a+(n_0-1)b)^{s_{n_0}}}\Big)\le -s_{n_0}<0
$$
as desired. Thus Lemma 2.8 is proved.
\end{proof}

\section{Proof of Theorem \ref{thm1}}

We can now prove Theorem \ref{thm1} as follows.\\

{\it Proof of Theorem \ref{thm1}.}
Obviously, $H_{a, b}(\mathcal{S}_n)>0$. So we need just to prove that
$H_{a, b}(\mathcal{S}_n)<1$ or $v_p(H_{a, b}(\mathcal{S}_n))<0$ for some prime $p$.
Let $d=\gcd(a,b)$. We divide the proof into the following four cases.

{\sc Case 1.} $a=b$.
By Bertrand's postulate, there is a prime $p\in (\frac{n}{2}, n]$,
which infers that
$p\le n<2p$.

If $a\ge p$, then $a>\frac{n}{2}$. By Lemma \ref{lem2.1},
$H_{a, b}(\mathcal{S}_n)<1$.

If $a<p$, then $\gcd(a, p)=1$. Since $p\le n<2p$, there is only one term
$ap$ divisible by $p$ in the finite arithmetic progression
$\{ak\}_{k=1}^{n}$. So by Lemma 2.8, we have $v_p(H_{a, b})<0$
as desired.

{\sc Case 2.} $a\ne b$ and $b=d$ or $2d$. Then Bertrand's postulate guarantees
the existence of a prime $p\in (\frac{a+(n-1)b}{2d}, \frac{a+(n-1)b}{d}]$.

{\sc Case 2.1.} $n<\frac{a}{b}+1$. Then $H_{a, b}(\mathcal{S}_n)<1$ by Lemma \ref{lem2.2}.

{\sc Case 2.2.} $d\ge p$. Then $d\ge p>\frac{a+(n-1)b}{2d}$. So, by Lemma
\ref{lem2.3}, $H_{a, b}(\mathcal{S}_n)<1$.

{\sc Case 2.3.} $n\ge \frac{a}{b}+1$ and $d<p$. Then $\frac{a+(n-1)b}{2d}
\ge \frac{a}{d}$ and $\gcd(d, p)=1$.

For $p\ge 3$, noticing that $\frac{b}{d}=1$ or $2$,
there is exactly one term divisible by $p$ in the finite
arithmetic progression $\{a+b(k-1)\}_{k=1}^{n}=
\{d(\frac{a}{d}+\frac{b}{d}(k-1))\}_{k=1}^{n}$
since $p\le \frac{a}{d}+\frac{b}{d}(n-1)<2p$.
Immediately, Lemma 2.8 gives us that $v_p(H_{a, b}(\mathcal{S}_n))<0$.

For $p=2$, one has $d=1$ since $d<p=2$, Since $a\ne b$ and $n\ge 2$,
we have $3\le a+b\le a+b(n-1)=\frac{a+b(n-1)}{d}<2p=4$.
Hence $a+b(n-1)=3$. This together with $n\ge \frac{a}{b}+1$
infers that $(n, a, b)=(2, 1, 2)$.
That is, $H_{a, b}(\mathcal{S}_n)=H_{1, 2}(\mathcal{S}_2)=1
+\frac{1}{3^{s_2}}$, which is obviously not an integer.

{\sc Case 3.} $a=1$ and $b\ge 3d$. Let $H_{a, b}(\mathcal{S}_n)=1+H'$,
where $H'=\sum\limits_{k=1}^{n-1}\frac{1}{(1+kb)^{s_{k+1}}}$.
Clearly, $H'>0$ for $n\ge 2$ and $H'=\frac{1}{(1+b)^{s_2}}<1$
is not an integer if $n=2$. So it is enough to prove that either
$H'<1$ or $v_p(H')<0$ for some prime $p$ when $n\ge 3$.
By Bertrand's postulate, there is a prime $p\in (\frac{n-1}{2}, n-1]$.

{\sc Case 3.1.} $n-1<\frac{1+2b}{b}e^{b(1-\frac{1}{1+b}
-\frac{1}{1+2b})}-\frac{1}{b}$. Then with (1) applied
to $a=1$, we derive that
\begin{align*}
H' &\le \sum\limits_{k=1}^{n-1}\frac{1}{1+kb} \\
&\le \frac{1}{1+b}+\frac{1}{1+2b}+\int_{2}^{n-1}\frac{1}{1+bt}dt \\
&=\frac{1}{1+b}+\frac{1}{1+2b}+\frac{1}{b}\log\frac{1+b(n-1)}{1+2b}<1
\end{align*}
as expected.

{\sc Case 3.2.} $n-1\ge \frac{1+2b}{b}e^{b(1
-\frac{1}{1+b}-\frac{1}{1+2b})}-\frac{1}{b}$.
Then
$$p>\frac{n-1}{2}\ge \frac{1+2b}{2b}e^{b(1-\frac{1}{1+b}-\frac{1}{1+2b})}
-\frac{1}{2b}. \eqno(5)$$
We claim that $p>3b$. Actually, if $b\ge 4$, then
$$b\Big(1-\frac{1}{1+b}-\frac{1}{1+2b}\Big)\ge b\Big(1-\frac{1}{1+4}
-\frac{1}{1+8}\Big)=\frac{31}{45}b\ge\frac{124}{45}.$$
But $e^x>\frac{135}{31}x$ for $x\ge\frac{124}{45}$. Therefore, by (5),
\begin{align*}
p&>\frac{1+2b}{2b}e^{\frac{31}{45}b}-\frac{1}{2b} \\
&>\frac{1+2b}{2b}\times \frac{135}{31}\times \frac{31}{45}b-\frac{1}{2b}\\
&=3b+\frac{3b-1}{2b}>3b
\end{align*}
as claimed. If $b=3$, then by (5) we have
$$p>\frac{1+2b}{2b}e^{b(1-\frac{1}{1+b}-\frac{1}{1+2b})}
-\frac{1}{2b}=\frac{7}{6}e^{\frac{51}{28}}-\frac{1}{6}\approx 7.044128639.$$
But $p$ is a prime. So $p\ge 11>9=3b$. The claim is proved.

Now from the claim, one concludes that $\gcd(p, b)=1$
and that there is exactly one term $1+k_1b$ (resp. $1+k_1b+pb$)
divisible by $p$ in $\{1+kb\}_{k=1}^{p}$ (resp. $\{1+kb\}_{k=p+1}^{2p}$),
where $1\le k_1\le p-1$. Moreover, by the claim one has
$$1+k_1b\le 1+(p-1)b<pb<\frac{1}{3}p^2 \eqno(6)$$
and
$$1+k_1b+pb<2pb<\frac{2}{3}p^2.  \eqno(7)$$
Therefore, $v_p(1+k_1b)=1=v_p(1+k_1b+pb)$.
Note that $1\le k_1\le p-1\le n-2$ since $p\le n-1$.
In the following we show that $v_p(H')<0$. Since
$p>\frac{n-1}{2}$, we have $n-1<2p<k_1+2p$. Thus there are at
most two terms divisible by $p$ in $\{1+kb\}_{k=1}^{n-1}$.

Now, if $n-1<k_1+p$, then there is exactly one term $1+k_1b$
divisible by $p$ in $\{1+kb\}_{k=1}^{n-1}=\{1+b+(k-1)b\}_{k=1}^{n-1}$.
Hence by Lemma 2.8, one has
$$v_p(H')=v_p\Big(\sum\limits_{k=1}^{n-1}
\frac{1}{(1+kb)^{s_{k+1}}}\Big)<0$$
as expected.

If $n-1\ge k_1+p$ and $s_{k_1+1}\ne s_{k_1+p+1}$, then
$$v_p(H')=\min\Big( v_p\Big(\dfrac{1}{(1+k_1b)^{s_{k_1+1}}}\Big),
v_p\Big(\dfrac{1}{(1+k_1b+pb)^{s_{k_1+p+1}}}\Big)\Big)<0$$
as desired.

If $n-1\ge k_1+p$ and $s_{k_1+1}=s_{k_1+p+1}:=s$, then let
$H'=A+B$, where
$$A:=\dfrac{1}{(1+k_1b)^s}+\dfrac{1}{(1+k_1b+pb)^s}
=\dfrac{(1+k_1b)^s+(1+k_1b+pb)^s}{(1+k_1b)^s(1+k_1b+pb)^s}$$
and
$$B:=\sum_{k=1\atop k\ne k_1, k\ne k_1+p}^{n-1} \dfrac{1}{(1+kb)^{s_{k+1}}}.$$
Evidently, one has $v_p(B)\ge 0$. On the other hand, by (6) and (7), one has
$$(1+k_1b)^s+(1+k_1b+pb)^s<\Big(\frac{1}{3}p^2\Big)^s
+\Big(\frac{2}{3}p^2\Big)^s\le p^{2s}$$
and  so $v_p((1+k_1b)^s+(1+k_1b+pb)^s)<2s$.
But $v_p((1+k_1b)^s(1+k_1b+pb)^s)=2s$. Therefore we have
$$v_p(A)=v_p((1+k_1b)^s+(1+k_1b+pb)^s)-v_p((1+k_1b)^s(1+k_1b+pb)^s)<0$$
and
$$v_p(H')=v_p(A+B)=\min(v_p(A), v_p(B))=v_p(A)<0$$
as required. So Theorem 1.1 is proved in this case.

{\sc Case 4.} $a\ne b$, $a\ge 2$ and $b\ge 3d$.
We revise the argument of Case 3. Clearly, we just need to deal with
the case when $H_{a, b}(\mathcal{S}_n)\ge 1$. In what follows,
let $H_{a, b}(\mathcal{S}_n)\ge 1$. Then by Lemma \ref{lem2.7},
there is a prime $p$ such that exactly one term in $\{a+(k-1)b\}_{k=1}^{n}$
is divisible by $p$, or $p\in(\max(\frac{n}{2}, \alpha), n]$
with $\alpha$ being given in (3).

If there is a prime $p$ such that exactly one term in
$\big\{a+(k-1)b\big\}_{k=1}^{n}$ is divisible by $p$,
then by Lemma 2.8, we have $v_p(H_{a, b}(\mathcal{S}_n))<0.$
If it doesn't hold, then by Lemma 2.7,
there is a prime $p\in(\max(\frac{n}{2}, \alpha), n]$.
Since $p>\alpha\ge 3b$, we have $\gcd(p,b)=1$
and in $\{a+kb\}_{k=0}^{p-1}$ (resp. $\{a+kb\}_{k=p}^{2p-1}$),
there is exactly one term $a+k_2b$ (resp. $a+k_2b+pb$)
divisible by $p$, where $0\le k_2\le p-1$. But $p-1\le n-1<2p-1$,
it implies that there are exactly two terms
$a+k_2b$ and $a+k_2b+pb$ in $\{a+(k-1)b\}_{k=1}^{n}$ divisible
by $p$. Since $0\le k_2\le p-1$,
we deduce that $a+k_2b\le pb+a-b$. If $a>b$, then
$p>\alpha=\frac{3b+\sqrt{9b^2+12(a-b)}}{2}$ from which
one derives that
$$pb+a-b<\frac{1}{3}p^2.  \eqno(8)$$
If $a<b$, then $a-b<0$ and $p>\alpha =3b$,
and so (8) still holds when $a<b$. Hence
$$p\le a+k_2b\le pb+a-b<\frac{1}{3}p^2. \eqno(9)$$
This infers that $v_p(a+k_2b)=1$. But $p>\alpha\ge 3b$.
Thus $pb<\frac{1}{3}p^2$ and by (9),
$$p\le a+k_2b+pb<\frac{2}{3}p^2.  $$
It follows that $v_p(a+k_2b+pb)=1$.

First, let $s_{k_2+1}\ne s_{k_2+p+1}$. Then
$$v_p(H_{a, b}(\mathcal{S}_n))=
\min\Big( v_p\Big(\dfrac{1}{(a+k_2b)^{s_{k_2+1}}}\Big),
v_p\Big(\dfrac{1}{(a+k_2b+pb)^{s_{k_2+p+1}}}\Big)\Big)$$
$$=\min(-s_{k_2+1},-s_{k_2+p+1})<0$$
as required.

Now let $s_{k_2+1}=s_{k_2+p+1}:=\bar s$. Then we split
$H_{a, b}(\mathcal{S}_n)$ into two parts: $H_{a, b}(\mathcal{S}_n)=C+D$,
where
$$C:=\dfrac{1}{(a+k_2b)^{\bar s}}+\dfrac{1}{(a+k_2b+pb)^{\bar s}}
=\dfrac{(a+k_2b)^{\bar s}+(a+k_2b+pb)^{\bar s}}{(a+k_2b)^{\bar s}(a+k_2b+pb)^{\bar s}}$$
and
$$D:=\sum_{k=1\atop k\ne k_2+1, k\ne k_2+p+1}^{n} \dfrac{1}{(a+(k-1)b)^{s_k}}.$$
However, from (8) and (9), one deduces that
$$0<(a+k_2b)^{\bar s}+(a+k_2b+pb)^{\bar s}<\Big(\frac{1}{3}p^2\Big)^{\bar s}
+\Big(\frac{2}{3}p^2\Big)^{\bar s}\le p^{2\bar s}.$$
So $v_p((a+k_2b)^{\bar s}+(a+k_2b+pb)^{\bar s})<2\bar s$.
But $v_p((a+k_2b)^{\bar s}(a+k_2b+pb)^{\bar s})=2\bar s$. Thus
$$v_p(C)=v_p((a+k_2b)^{\bar s}+(a+k_2b+pb)^{\bar s})
-v_p((a+k_2b)^{\bar s}(a+k_2b+pb)^{\bar s})<0.$$
Since $v_p(D)\ge 0$, we deduce immediately that
$$v_p(H_{a, b}(\mathcal{S}_n))=v_p(C+D)=v_p(C)<0$$
as one desires.

This concludes the proof of Theorem \ref{thm1}.
\hfill$\Box$

\section{Final remarks}
Let $a, b$ and $n$ be positive integers. For any integer $k$ with
$1\le k\le n$, $H_{a, b}^{(k)}(\mathcal{S}_n)$ stands for the
$k$-th elementary symmetric function of the $n$ fractions:
$\frac{1}{a^{s_1}}, \frac{1}{(a+b)^{s_2}}, ..., \frac{1}{(a+(n-1)b)^{s_n}}$.
Namely,
$$
H_{a, b}^{(k)}(\mathcal{S}_n):=\sum_{1\le i_1<...<i_k\le n}
\prod_{j=1}^{k}\frac{1}{(a+b(i_j-1))^{s_{i_j}}}.
$$
Then $H_{a, b}^{(1)}(\mathcal{S}_n)=H_{a, b}(\mathcal{S}_n)$.
Hong and Wang \cite{[HW]} showed that if all $s_i$ are equal to 1,
then all $H_{a, b}^{(k)}(\mathcal{S}_n)$ are not integers if $n\ge 4$.
We believe that such result is still true for any infinite sequence
$\mathcal{S}$ of positive integers. Namely, we propose the
following conjecture.\\
\\
{\bf Conjecture 4.1.} {\it For any infinite sequence $\mathcal{S}$
of positive integers and arbitrary positive integers $a, b$ and $n$,
if $n\ge 4$, then none of
$H_{a, b}^{(1)}(\mathcal{S}_n), H_{a, b}^{(2)}(\mathcal{S}_n),
..., H_{a, b}^{(n)}(\mathcal{S}_n)$ is an integer.}\\

By Theorem 1.1, one knows that Conjecture 4.1 is true when
$k=1$. It is clear that Conjecture 4.1 holds when $k=n$.
Hence we need just to look at the case $2\le k\le n-1$.
On the other hand, if all $s_i$ are equal to 1, then Hong and
Wang's result \cite{[HW]} says that Conjecture 4.1 is true.
If all $s_i$ are greater than 2, then one can show the truth
of Conjecture 4.1. However, if there exist indexes $i$ and $j$
such that $s_i=1$ and $s_j\ge 2$, then the situation becomes
complicated and hard, and so the truth of Conjecture 4.1
is still kept open so far.\\

\begin{center}
{\sc Acknowledgement}
\end{center}
The authors would like to thank the anonymous referee for
careful reading of the manuscript and helpful comments.

\bibliographystyle{amsplain}

\end{document}